\documentclass[a4paper, 12pt, DIV=11]{scrartcl}
\usepackage[utf8]{inputenc}
\usepackage[T1]{fontenc}
\usepackage[british]{babel}
\usepackage{graphicx, xcolor}
\usepackage{booktabs}
\usepackage{enumerate}
\usepackage{amsmath,amsthm,amsfonts,amssymb}

\usepackage{libertine}
\usepackage[libertine]{newtxmath}
\usepackage[lining]{FiraSans}
\setkomafont{disposition}{\firamedium}

\usepackage[colorlinks=true, linkcolor=blue, filecolor=magenta, urlcolor=cyan, citecolor=gray, unicode]{hyperref}
\usepackage{cleveref}
\crefname{equation}{}{}
\crefname{chapter}{Chapter}{Chapters}
\crefname{section}{Section}{Sections}
\crefname{subsection}{Section}{Sections}
\crefname{subsubsection}{Section}{Sections}
\crefname{figure}{Figure}{Figures}
\crefname{table}{Table}{Tables}
\crefname{algorithm}{Algorithm}{Algorithms}

\numberwithin{equation}{section}
\theoremstyle{definition}

\crefname{question}{Question}{Questions}
\newtheorem{definition}{Definition}[section]
\crefname{def}{Definition}{Definitions}

\crefname{ex}{Example}{Examples}

\theoremstyle{plain}
\newtheorem{thm}[definition]{Theorem}
\crefname{thm}{Theorem}{Theorems}
\newtheorem{conj}{Conjecture}[section]
\crefname{conj}{Conjecture}{Conjectures}
\newtheorem{lemma}[definition]{Lemma}
\crefname{lemma}{Lemma}{Lemmas}

\crefname{cor}{Corollary}{Corollaries}

\crefname{prop}{Proposition}{Propositions}

\crefname{obs}{Observation}{Observations}

\crefname{claim}{Claim}{Claims}

\theoremstyle{remark}

\crefname{rmk}{Remark}{Remarks}

\newcommand{\setbuilder}[2]{\left\{#1\ \colon #2\right\}}

\newcommand{\N}{\mathbb{N}}

\newcommand{\cE}{\mathcal{E}}

\newcommand{\cH}{\mathcal{H}}

\newcommand{\cK}{\mathcal{K}}

\newcommand{\card}[1]{\left|#1 \right|}
\renewcommand{\phi}{\varphi}
\renewcommand{\subset}{\subseteq}
\title{\LARGE Partitioning infinite hypergraphs into few monochromatic Berge-paths}
\author{Sebasti\'an Bustamante \and Jan Corsten\thanks{Department of Mathematics, London School of Economics, Houghton St, WC2A 2AE London, Email: \href{mailto:j.corsten@lse.ac.uk}{j.corsten@lse.ac.uk}, \href{mailto:n.frankl@lse.ac.uk}{n.frankl@lse.ac.uk}.} \and N\'ora Frankl\footnotemark[1]}

\begin{document}
\maketitle
\begin{abstract}
Extending a result of Rado to hypergraphs, we prove that for all $s, k, t \in \mathbb{N}$ with $k \geq t \geq 2$, the vertices of every $r = s(k-t+1)$-edge-coloured countably infinite complete $k$-graph can be partitioned into the cores of at most $s$ monochromatic $t$-tight Berge-paths of different colours. We further describe a construction showing that this result is best possible.
\end{abstract}

\section{Introduction}
Lehel's conjecture (first seen in \cite{Aye79}) states that the vertices of every $2$-edge-coloured complete graph can be partitioned into two monochromatic cycles, one of each colour. Here, single vertices and edges are considered to be cycles and this convention is used throughout this paper.
The conjecture was proved for very large graphs by \L uczak, R\"odl and Szemer\'edi \cite{LRS98} in 1998, for large graphs by Allen \cite{All08} in 2008, and for all graphs by Bessy and Thomass\'e \cite{BT10} in 2010.

Erdős, Gyárfás and Pyber \cite{EGP91} conjectured in 1991 that this can be extended to $r$ colours (allowing $r$ monochromatic cycles). 
This was however disproved by Pokrovskiy \cite{Pok14}, who showed that for every $r \geq 3$, there are infinitely many $r$-edge-coloured complete graphs whose vertices cannot be covered by $r$ monochromatic vertex-disjoint cycles.
Finding the minimum number of monochromatic vertex-disjoint cycles needed to cover the vertices of any $r$-edge-coloured complete graph remains a big open problem. We note that a priori it is not obvious that this minimum is independent of the size of the complete graph we wish to cover. The fact that this is the case was proved by Erdős, Gyárfás and Pyber \cite{EGP91}, who also presented a simple construction in which $r$ cycles are needed. The currently best-known upper bound is $100r\log r$, due to Gy\'arf\'as, Ruszink\'o, S\'ark\"ozy and Szemer\'edi \cite{GRS+06}.

An infinite analogue of the conjecture of Erd\H os, Gy\'arf\'as and Pyber is true as proved by Rado \cite{Rad78} already in 1978.

\begin{thm}[\cite{Rad78}]\label{thm:rado}
The vertices of every countably infinite $r$-edge-coloured complete graph can be partitioned into $r$ monochromatic paths (infinite or finite), one of each colour.
\end{thm}

 Rado's theorem is best possible, as the construction for finite graphs in  \cite{EGP91} can be extended to infinite graphs. 

In this note we consider extensions of this result to hypergraphs.
A $k$-uniform hypergraph (or shortly, a $k$-graph) is a tuple $\cH = (V,\cE)$ where $V$ and $\cE$ are sets with $\cE \subseteq \binom{V}{k}$. The complete $k$-graph with vertex-set $V$, denoted by $\cK_{V}^{(k)}$, is the $k$-graph with edge-set $\cE = \binom{V}{k}$.

There are different notions of paths in hypergraphs, \emph{loose paths}, \emph{tight paths} and \emph{Berge-paths}; all of these coincide with the notion of paths when $k=2$. In this note, we will mainly consider Berge-paths.
Given integers $2 \leq t \leq k$ and $ \ell \geq 1$, a finite ($k$-uniform) {\itshape $t$-tight Berge-path} of length $\ell$ is a pair $(X,F)$ defined as follows. $X = \{v_1,\dots, v_{\ell+t-1}\} \subseteq V$ is a set of $\ell + t -1$ distinct vertices, $F = \{e_1, \dots, e_{\ell}\} \subseteq \cE$ is a set of $\ell$ distinct edges and $v_i,v_{i+1}, \dots, v_{i+t-1} \in e_i$ for all $i \in [\ell]$. For technical reasons we define a $t$-tight Berge-path of length $0$ as a pair $(\{v\}, \emptyset)$. A (one-way) infinite $t$-tight Berge-path is a pair $(X,F)$ where $X = \{v_i: i \in \N \} \subseteq V$, $F = \{e_i: i \in \N \} \subseteq \cE$, and $v_i,v_{i+1}, \dots v_{i+t-1} \in e_i$ for all $i \in \N$. For a $t$-tight Berge-path $P = (X,F)$, $X$ is called the {\itshape core} of $P$. A family $P_1 = (X_1, F_1), \dots, P_r = (X_r, F_r)$ of $t$-tight Berge-paths {\itshape core-partitions} $V$ if $X_1 \cup \dots \cup X_r$ is a partition of $V$. Given an edge-colouring $\varphi$ of a $k$-graph $\cH$, a $t$-tight Berge-path $P = (X,F)$ in $\cH$ is said to be {\itshape monochromatic} in colour $c$ if $\varphi(f) = c$ for all $f \in F$.

A $2$-tight Berge-path is simply called a Berge-path and a $k$-tight Berge-path is called a tight path. The $k$-uniform \emph{loose path} of length $\ell$ consists of $n=k(\ell -1)+1$ vertices $\{v_1, \ldots, v_n\}$ and the $\ell$ edges $\{v_{i(k-1)+1}, \ldots, v_{i(k-1)+k}\}$ for $i = 0, \ldots, \ell -1$. The \emph{infinite loose path} consists of the vertices $\{v_1, v_2, \ldots\}$ and edges $\{v_{i(k-1)+1}, \ldots, v_{i(k-1)+k}\}$ for all $ i =0,1,\ldots$.

%Berge-cycles are defined in a similar way.

Many extensions of path partition problems to hypergraphs have been studied, considering loose paths \cite{GS13, GS14, Sar14}, tight paths \cite{ESS+17, bustamante2019} and Berge-paths \cite{GS13}.
Most relevant for the topic of this note are the following two extensions of \cref{thm:rado}.

\begin{thm}[Gy\'arf\'as--S\'ark\"ozy \cite{GS13}]\label{thm:loose}
The vertices of every countably infinite $r$-edge-coloured complete $k$-graph can be partitioned into $r$ monochromatic loose paths (infinite or finite), one of each colour.
\end{thm}

\begin{thm}[Elekes--Soukup--Soukup--Szentmikl\'ossy \cite{ESS+17}]\label{thm:tight}
The vertices of every countably infinite $r$-edge-coloured complete $k$-graph can be partitioned into $r$ monochromatic tight paths (infinite or finite), one of each colour.
\end{thm}

The latter result answers a question of Gy\'arf\'as and S\'ark\"ozy from \cite{GS13}.
Note that both \cref{thm:loose,thm:tight} reduce to \cref{thm:rado} when $k=2$. Our main result extends \cref{thm:rado} in a similar way to Berge-paths. It turns out that $\lceil r/(k-1) \rceil$ paths suffice.

\begin{thm}\label{BCF18:thm-simple}
	For all $s,k \in \N$ with $k \geq 2$ and every $r=s(k-1)$-edge-colouring of $\cK_{\N}^{(k)}$, the vertices can be core-partitioned into $s$ monochromatic Berge-paths of different colours.
\end{thm}

Note that \cref{BCF18:thm-simple} reduces to \cref{thm:rado} when $k =2$ as well.  We shall actually prove the following more general result about $t$-tight Berge-paths.

\begin{thm}\label{BCF18:thm-1}
	For all $s,k,t \in \N$ with $k \geq t \geq 2$ and every $r = s(k-t+1)$-edge-colouring of $\cK_{\N}^{(k)}$, the vertices can be core-partitioned into $s$ monochromatic $t$-tight Berge-paths of different colours.
\end{thm}

Note that the case $k=2$ reduces to \cref{BCF18:thm-simple}, and the case $k=t$ reduces to \cref{thm:tight}. The following theorem shows that \cref{BCF18:thm-1} is best possible.

\begin{thm}\label{BCF18:thm-2}
	For all $s,k,t \in \N$ with $k \geq t \geq 2$, there is an edge-colouring of $\cK_{\N}^{(k)}$ with $r = s(k-t+1) + 1$ colours in which the vertices cannot be covered by the cores of $s$ monochromatic $t$-tight Berge-paths.
\end{thm}

We will prove \cref{BCF18:thm-2} in \cref{sec:constr} and \cref{BCF18:thm-1} in Section~\ref{BCF18:sec-ub}. 

\section{Proof of \texorpdfstring{\cref{BCF18:thm-2}}{Theorem 1.6}}\label{sec:constr}

The construction described in the proof generalises the construction from \cite{EGP91}.

\begin{proof}[Proof of \cref{BCF18:thm-2}]
  We denote the lexicographical ordering of $\binom{[r]}{s}$ by $\prec$. First, we partition $\N$ into sets $\left \{B_I: I \in \binom{[r]}{s} \right \}$ so that all $B_I$'s but $B_{\{r-s+1, \dots, r\}}$ are finite and $|B_I| \geq st \cdot \sum_{J \prec I} (|B_J|+1)$ for every $I \in \binom{[r]}{s}$.\footnote{This growth rate of the $B_{|I|}$ can be improved by a more careful analysis.} For $x \in \N$, let $I(x)$ be the $s$-subset of $[r]$ for which $x \in B_{I(x)}$. We define an $r$-edge-colouring $\varphi$ of $\cK_{\N}^{(k)}$ as follows.
	
	For every $e \in E(\cK_{\N}^{(k)})$ we consider an order $x_e^1, \dots, x_e^k$ of $e$ satisfying $I(x_e^i) \preceq I(x_e^j)$ for all $1 \leq i < j \leq k$, and define $\varphi(e)$ as an arbitrary member of $[r] \setminus \bigcup_{i \leq k-t+1} I(x_e^i)$.  
	
	Assume for contradiction that there are monochromatic $t$-tight Berge-paths $P_1, \dots, P_s$ with cores $X_1, \dots, X_s$  so that $\bigcup_i X_i = \N$ and let $C \subset [r]$ be a set of size $s$ which contains all colours used by these $t$-tight Berge-paths.
	
Observe that for every edge $e$ with $e \cap B_C \neq \emptyset$ and $\varphi(e) \in C$ we have 
\begin{equation}\label{eq:constr}
    \left |e \cap \bigcup\nolimits_{J \prec C } B_J \right | \geq k-t+1.
\end{equation}
Indeed, if $\left | e \cap \bigcup_{J \prec C} B_J \right | < k-t+1$, then $C \preceq I(x_e^{k-t+1})$ and thus $\varphi(e) \notin C$. 

For $ i \in C$, let $F_i$ be the set of all $f \in \binom{X_i}{t}$ which consist of $t$ consecutive vertices of $X_i$ with at least one element in $B_C$. Let $f \in F_i$ and let $e \in E(P_i)$ be some edge with $f \subset e$. By \cref{eq:constr}, we have $\card{e \setminus \bigcup_{J \prec C} B_J} \leq t-1$ and therefore some vertex in $f$ must be in $\bigcup_{J \prec C} B_J$.
Since every $ v \in \N$ is contained in at most $t$ sets $f \in F_i$, it follows that 
\begin{equation}\label{eq:constr2}
  \card{F_i} \leq t \card{\bigcup\nolimits_{J \prec C} B_J}.
\end{equation}

Observe now that for all but at most $t-1$ vertices $ v \in X_i \cap B_C$, there is a unique $f \in F_i$ starting at $v$ and thus
\begin{equation}\label{eq:constr3}
  \card{X_i \cap B_C} \leq \card{F_i} + t - 1.
\end{equation}
Combining \cref{eq:constr2,eq:constr3}, we get
\[ |X_i \cap B_C| \leq t \cdot \sum\nolimits_{J \prec C} |B_J| +t-1 < |B_C|/r \] for every $i \in [r]$ and hence $|B_C| = \left | B_C \cap \left ( \bigcup_i X_i \right ) \right | < |B_C|$, a contradiction.
\end{proof}

%A simple modification of the argument yields the following result which shows that Conjecture~\ref{BCF18:conj} is best possible if it is true.

%\begin{thm}\label{BCF18:thm-3}
%	For all $c,r,k \in \N$ with $k \geq 2$, there is some $n_0 = n_0(c,r,k)$ such that the following is true for every natural number $n \geq n_0$. There is an edge-colouring of $\cK_{\N}^{(k)}$ in which the cores of any $r$ monochromatic Berge-cycles can cover at most $n-c$ vertices.
%\end{thm}

%\begin{proof}
%	If for the $B_I$'s instead of $|B_I|> r \cdot \sum_{J<_l I}|B_J|$ we require $|B_I|> c + r \cdot \sum_{J<_l I}|B_J|$, we obtain the desired construction.
%\end{proof}

\section{Proof of \texorpdfstring{\cref{BCF18:thm-1}}{Theorem 1.5}}\label{BCF18:sec-ub}

Our proof is based on ideas from \cite{ESS+17}. First, we need to introduce some notation.
An $r$-multi-colouring of a $k$-graph $G$ is a function $\chi: E(G) \to 2^{[r]}$. Given a set $F \subset E(G)$, we denote by $\chi(F) = \bigcap_{e \in F} \chi(e)$ the set of colours they have in common and say that $F$ is ($\chi$-)\emph{monochromatic} if $\chi(F)$ is non-empty.
For a given $r$-colouring $\varphi$ of $\cK := \cK_{\N}^{(k)}$ and $i,j \in \N$ with $j < k$, we define an $r$-multi-colouring $\varphi_{i,j}: \binom{\N \setminus \{i\} }{j} \to 2^{[r]}$ by
\[\varphi_{i,j}(f) = \{\varphi(e): e \in E(\cK) \text{ and } \{i\} \cup f \subseteq e \}.\]

Furthermore, we call $\{ \cK_i : i \in \N \}$ a \emph{$j$-clique-chain} w.r.t.\ an $r$-colouring $\varphi$ of $\cK$ if $\cK_1$ is a $\varphi_{1,j}$-monochromatic copy of $\cK_{\N}^{(j)}$ with $V(\cK_1) \subseteq \N$ and $\cK_{i}$ is a $\varphi_{i,j}$-monochromatic copy of $\cK_{\N}^{(j)}$ with $V(\cK_i) \subseteq V(\cK_{i-1})$ for every $ i \in \N$. 

Observe that, by Ramsey's theorem \cite{Ram30} for infinite hypergraphs, there exists a $j$-clique-chain for every $r$-colouring of $\cK$ and every $j \in [k-1]$.

For a $j$-clique-chain $\{\cK_i : i \in \N \}$ we define a vertex-multi-colouring $\chi: \N \to 2^{[r]}$ by  $\chi(i)= \bigcap_{e \in E(\cK_i)} \varphi_{i,j}(e)$ for every $i\in \N$. We call $\chi$ the \emph{clique-colouring induced by $\{\cK_i : i \in \N \}$}.
%\textcolor{red}{Should we define it as a multicolouring of all available colours? This is about Peter's 4th remark.}

\begin{lemma}\label{BCF18:lem-1}
For all $s,k,t \in \N$ with $k \geq t \geq 2$ and every $r =s(k-t+1)$-colouring of $\cK_{\N}^{(k)}$, there is a $(t-1)$-clique-chain that induces a clique-colouring using at most $s$ colours.
\end{lemma}

\begin{proof}
Let $\varphi$ be the given $r$-colouring of $\cK := \cK_{\N}^{(k)}$. Furthermore, let $C_1\cup\dots\cup C_{r}$ be a partition of the set of $r = s(k-t+1)$ colours into $r$ blocks of size $s$. We will show that there is a $(t-1)$-clique-chain and some $i\in [r]$ such that for the induced clique-colouring we have $\chi(v)\cap C_i\ne \emptyset$ for all $v\in \mathbb{N}$.

We call an infinite $(t-1)$-uniform clique $\cK'$ \emph{maximally-monochromatic} w.r.t.\ a multi-colouring $\psi$ of $\cK'$ and a set $C\subseteq[r]$ if there is no infinite clique $\cK'' \subset \cK'$ with $|\setbuilder{i}{\psi(\cK'')\cap C_i\ne \emptyset}\cap C| > |\setbuilder{i}{\psi(\cK')\cap C_i\ne \emptyset}\cap C|$. Note that a maximally-monochromatic clique is not necessarily monochromatic (since all its infinite monochromatic subcliques might have colours not in $C$). Further note that every infinite clique contains a maximally-monochromatic infinite clique (since $r$ is finite).
    
We build a $(t-1)$-clique-chain as follows.
Let $\cK_1$ be any $\varphi_{1,t-1}$-monochromatic, maximally-monochromatic $(t-1)$-uniform clique w.r.t. $\phi_{1,t-1}$ and $D=[r]$, and let $D_1:=\setbuilder{i}{\phi_{1,t-1}(\cK_1)\cap C_i\ne \emptyset}$. Now, for every $j \in \N$, let $\cK_{j+1}$ be a $\varphi_{j+1,t-1}$-monochromatic, maximally-monochromatic clique w.r.t.\ $\phi_{j+1,t-1}$ and $D_i$ with $V(\cK_{j+1}) \subset V(\cK_j)$ and let $D_{i+1} = \setbuilder{i}{\phi_{j+1,t-1}(\cK_{j+1}) \cap C_i\ne \emptyset}$. If there is some $i \in [r]$ such that $C_i\cap D_j\ne \emptyset$ for all $j\in \mathbb{N}$, then $\{\cK_1, \cK_2, \ldots\}$ is a $(t-1)$-clique-chain with the desired property. Hence we may assume that there is no such $i$.

Thus, there exist $j_1, \ldots, j_r$, such that $C_i \cap D_{j_i} = \emptyset$
but $C_i \cap D_{i_j -1} \not = \emptyset$ for every $i \in [r]$. Without loss of generality we may assume that $j_1 \leq \ldots \leq j_{r}$.
Let $X = V(\cK_{j_{r}})$ and note that $V(\cK_{j_i}) \supseteq X$ for every $i \in [r]$. Define $\Phi: \binom{X}{t-1} \to 2^{[r]}$ by
\[ \Phi(f) = \{ \phi(e): e \in E(\cK) \text{ and } \{j_1, \ldots, j_{r}\} \cup f \subset e\}.\]
Note that every $ f \in \binom{X}{t-1}$ receives at least one colour,
%(since $q+t-1 = k$),
and that $\Phi(f) \subset \phi_{j_i,t-1}(f)$ for every $f \in \binom{X}{t-1}$ and every $i \in [r]$.
By Ramsey's theorem for hypergraphs there is a $\Phi$-monochromatic infinite clique $\cK'$ in $X$. Therefore, there is some $\ell\in [r]$ such that $\Phi(\cK')\cap C_{\ell}\ne \emptyset$ and consequently $\cK_{j_\ell}$ is not maximally monochromatic.
%	
%Assuming the contrary, for every $(t-1)$-clique-chain $\{ \cK_i : i \in \N \}$ there are vertices $v_1, \dots, v_{k-t+1} \in \N$ for which every monochromatic copy of $\cK_{\N}^{(t-1)}$ under $\varphi_{v_i}$ and contained in $\cK_{v_i - 1}$ is not coloured with any $c \in C_i$. Define $\varphi_*: \binom{ \N \setminus \{v_1, \dots, v_{k-t+1}\}}{t-1}  \to 2^{[r]}$ by $$\varphi_*(f) = \varphi(f \cup \{v_1, \dots, v_{k-t+1} \}),$$ and note that $\varphi_*(f) \in \varphi_{v_i}(f)$ for every $f \in  \binom{ \N \setminus \{v_1, \dots, v_{k-t+1}\}}{t-1}$.
%	
%By the infinite version of Ramsey's theorem, there exists a monochromatic copy $\cK^{(t-1)}$ of $\cK_{\N}^{(t-1)}$ under $\varphi_*$ coloured by $c \in C_i$ for some $i \in [k-t+1]$, with its vertices contained in $V(\cK_{v_{k-t+1}})$. This is a contradiction since the vertices of $\cK^{(t-1)}$ are also contained into $V(\cK_{v_i-1})$ and $\cK_{v_i-1}$ contains no monochromatic copy of $\cK_{\N}^{(t-1)}$ with colours in $C_i$.
\end{proof}

We proceed now with the proof of Theorem~\ref{BCF18:thm-1}.

\begin{proof}[Proof of Theorem~\ref{BCF18:thm-1}]
	Let $\varphi$ be the given $r$-colouring of $\cK = \cK_{\N}^{(k)}$. By Lemma~\ref{BCF18:lem-1}, there is a $(t-1)$-clique-chain $\{ \cK_i : i \in \N \}$ that induces a clique-colouring $\chi$ using at most $s$ colours (without loss of generality these colours are $1, \ldots, s$).
%and let $\chi: \N \to [s]$ be its associated clique-colouring.
For $i \in [s]$, let $A_i \subset \N$ be the set of vertices of colour $i$ according to $\chi$.
    
%Without loss of generality, there is some $r' \in \{0, \ldots, r\}$ so that $A_1,\ldots,A_{r'}$ are finite and $A_{r'+1},\ldots, A_r$ are infinite. Let $X := \bigcup_{i=1}^{r'} A_i$ and let $Y = \emptyset$ (we will use $X$ and $Y$ to denote already used vertices and edges). We begin by covering all finite $A_i$ by finite Berge-paths. Assume that $r' \geq 1$ (otherwise, we may skip to the next step).

%For each $i \in [r]$, first do the following choose distinct vertices $b_2, \ldots, b_{r-1},b \in V(\cK_{a_2}) \setminus X$ arbitrary and add them to $X$. Let $ f_1 = \{b_1, \ldots, b_{r-1},b\}$ and $f_2 = \{b_2, \ldots, b_{r-1}, b, a_2\}$ and note that, since $1 \in \phi_{a_2,r-1}(\cK_{a_2}) $, we have $1 \in \phi_{a_2,r-1}(f_i)$ for both $i =1,2$. Hence there are edges $e_1,e_2 \in E(\cK)$ of colour $1$ with $f_1 \subset e_1$ and $f_2 \subset e_2$. If $e_1 \not = e_2$, let $b_{r} := b$ and $b_{r+1} = a_2$ and add $b$ to $X$ and $e_1,e_2$ to $Y$. Note that $e_1$ and $e_2$ are the edges for the first two consecutive $t$-tuples $f_1,f_2$ in the Berge-path we are constructing. If $e_1 = e_2$, let $b_{r} = a_2$ and add $e_1$ to $Y$. Note that the first consecutive $t$-tuple $\{b_1, \ldots, b_t \}$ is contained in $e_1$ and thus $e_1$ can be chosen as the edge for $\{b_1, \ldots, b_t\}$ in the Berge-path we are constructing
   
By repeating the following process 
%forever
we will simultaneously build $t$-tight monochromatic Berge-paths $P_1, \ldots, P_r$ with core-vertex sequences $\{b_{i,1}, b_{i,2}, \ldots\}$ for every $i \in [s]$.
Let $b_{i,1} := \min A_i$ for every $i \in [s]$. In every step, we will add to each path $t$ or $t-1$ vertices making sure that for every $i \in [s]$, the last new vertex, say $b_{i,n_i}$, is in $A_i$, and that the other new vertices are in $V(\cK_{b_{i,n_i}})$.
Right after choosing the vertex $b_{i,j}$, we will choose  a unique edge $e_{i,j} \in E(\cK)$ of colour $i$ which contains the $t$ consecutive vertices $b_{i,j-t+1}, \ldots, b_{i,j}$ for every $j \geq t$ and $i \in [s]$. Let $X =\{b_{1,1}, \ldots, b_{s,1}\}$ and let $Y = \emptyset$. We will use $X$ to keep track of already used vertices and $Y$ to keep track of already used edges.

For each $i \in [s]$ do the following.\footnote{To avoid unnecessary subscripts for `local variables', we treat $i$ as being fixed in the following.}
Suppose the current path $P_i$ ends in $b_{i,n} \in A_i$ for some $n \in \N$.
We will now extend $P_i$ by $t$ or $t-1$ vertices as follows. Let $a$ be the smallest vertex in $A_i \setminus X$ (if $A_i\setminus X$ is empty, the path $P_i$ is complete and we move to the next step).
%We will extend $P_i$ by $t$ or $t-1$ consecutive vertices $(b_{n+1},\dots,b_{i,n+t-2},b_{i,n+t-1}=b,b_{i,n+t}=a)$ or $(b_{n+1},\dots,b_{i,n+t-2},b_{i,n+t-1}=a)$ as follows.
Add $a$ to $X$ and do the following for every $j=1, \ldots,t-2$. Choose a vertex $b_{i,n+j} \in V(\cK_{a}) \setminus (\bigcup Y)$ and add it to $X$ (note that this is always possible since $V(\cK_{a})$ is infinite and $Y$ is finite).
Let $f_{i,n+j} = \{b_{i,n+j-t+1}, \ldots, b_{i,n+j}\}$ be the set of the $t$ consecutive vertices in the core of the Berge-path $P_i$ ending at $b_{i,n+j}$. Note that $f_{i,n+j} \setminus \{b_{i,n}\} \in E(\cK_{b_{i,n}})$ and thus $ i \in \phi_{b_{i,n},t-1}(f_{i,n+j} \setminus \{b_{i,n}\})$. Hence, by the definition of $\phi_{b_{i,n},t-1}$, there exist $e_{i,n+j} \in E(\cK)$ with $ f_{i,n+j} \subset e_{i,n+j}$ and $\phi(e_{i,n+j})=i$. Add $e_{i,n+j}$ to $Y$. Since $b_{i,n+j} \not \in \bigcup Y$, we have $e_{i,n+j} \not \in Y$ and we can therefore use $e_{i,n+j}$ as the edge for $f_{i,n+j}$ in our Berge-path.

Choosing the next vertex will be slightly more complicated (since $a$ might be in some edge in $Y$). Let $b \in V(\cK_{a}) \setminus \bigcup Y$ and let $f_1' = \{b_{i,n}, \ldots, b_{i,n+t-2},b\}$ and $f_2' = \{b_{i,n+1}, \ldots, b_{i,n+t-2}, b, a\}$, and note that $f_1' \setminus \{b_{i,n}\} \in E(\cK_{b_{i,n}})$ and $f_2' \setminus \{a\} \in E(\cK_{a})$.
As before, $ i \in \phi_{b_{i,n},t-1}(f_1' \setminus \{b_{i,n}\}) \cap \phi_{a,t-1}(f_2' \setminus \{a\})$ and thus there exist $e_1',e_2' \in E(\cK)$ with $ f_s' \subset e_s'$ and $\phi(e_s')=i$ for both $s = 1,2$.
If $e_1' \not = e_2'$, let $b_{i,n+t-1} := b$ and $b_{i,n+t} = a$ and let $e_{i,n+t-1} := e_1'$ and $e_{i,n+t} = e_2'$.
Add $b_{i,n+t-1}$ to $X$ and $e_{i,n+t-1},e_{i,n+t}$ to $Y$. Note that $e_{i,n+t-1}$ and $e_{i,n+t}$ can be chosen as the edges for the $t$ consecutive vertices of $P_i$ ending in $b_{i,n+t-1}$ and $b_{i,n+t}$.
If $e_1' = e_2'$, let $b_{i,n+t-1} = a$ and $e_{i,n+t-1} = e_1' = e_2'$, and add $e_{i,n+t-1}$ to $Y$. Note that the $t$ consecutive vertices of $P_i$ ending in $b_{i,n+t-1}$ are contained in $e_{i,n+t-1}$ and $e_{i,n+t-1} \not \in Y$. Hence, $e_{i,n+t-1}$ can be chosen as the edge for the $t$ consecutive vertices of $P_i$ ending in $b_{i,n+t-1}$.

By construction, $P_1, \ldots, P_s$ are monochromatic $t$-tight Berge-paths whose cores are disjoint. Furthermore, since at the beginning of every step the smallest uncovered vertex $a$ of $A_i$ is chosen, we have $ \bigcup_i V(P_i) = \N$. 
\end{proof}

\section{Further remarks and open problems}
\cref{thm:rado,thm:loose,thm:tight,BCF18:thm-simple} remain true when we consider cycles instead of paths, where an infinite cycle is a two-way infinite path.
It is not clear to us however if one can replace paths by cycles in \cref{BCF18:thm-1} when $2 < t < k$. Difficulties only arise when trying to close finite paths to cycles, hence we can replace paths by cycles if we allow finitely many vertices to be uncovered.

A natural question to ask is if similar results hold in the finite setting. 
Gy\'arf\'as, Lehel, S\'ark\"ozy and Schelp \cite{GLS+08} conjectured that every finite $(k-1)$-edge-coloured complete $k$-graph contains a monochromatic Hamiltonian Berge-cycle. Note that, in the infinite setting, this is a special case of \cref{BCF18:thm-simple}.
After partial results in \cite{GLS+08, GSS10a, GSS10b, MO17}, Omidi \cite{Omi14} announced a proof of this conjecture.

We believe that a generalisation of this to more colours, similar as in \cref{BCF18:thm-simple}, is true as well.

\begin{conj}\label{BCF18:conj1}
	For all $s,k \in \N$ with $k \geq 2$, there is some $c = c(s,k) \in \N$ such that the following is true for all $n \in \N$. In every $r = s(k-1)$-edge-colouring of $\cK_{n}^{(k)}$, there is a collection of at most $s$ monochromatic $t$-tight Berge-cycles whose cores are disjoint and cover all but $c$ vertices.
\end{conj}

For $k=2$, this reduces to a conjecture of Pokrovskiy \cite{Pok14}.
We further believe that this can be extended to $t$-tight Berge-cycles similarly to \cref{BCF18:thm-1}.

\begin{conj}\label{BCF18:conj2}
	For all $s,k,t \in \N$ with $k \geq t \geq 2$, there is some $c = c(s,k,t) \in \N$ such that the following is true for all $n \in \N$. In every $r = s(k-t+1)$-edge-colouring of $\cK_{n}^{(k)}$, there is a collection of at most $s$ monochromatic $t$-tight Berge-cycles whose cores are disjoint and cover all but $c$ vertices.
\end{conj}

A simple modification of the construction in \cref{sec:constr} shows that these conjectures are best possible (if true) apart from the finite leftover.

Recently we learned that Gerbner, Methuku, Omidi and Vizer \cite[unpublished]{gerbner2019} made some progress towards these conjectures.

% Gerbner, Methuku, Omidi and Vizer (D\'aniel Gerbner, private communication) proved the following weaker version. If $f(r)$ denotes the minimum number of cycles needed to cover the vertices of an $r$-edge-coloured complete graph, then $\lfloor k/2 \rfloor c$ Berge-cycles can core partition the vertices of any $c$-edge-coloured complete $k$-graph.

\section*{Acknowledgments}
The authors would like to thank Peter Allen and Jan van den Heuvel for their helpful suggestions on the manuscript. The first author was supported by CONICYT Doctoral Fellowship 21141116. Part of this research was done while the second and third author visited Universiad de Chile with the support of the Santander Travel Research Fund.

\bibliographystyle{amsplain}
\bibliography{bib.bib}

\end{document}